\documentclass[12pt]{amsart}
\usepackage{amssymb}
\usepackage{amsthm}
\usepackage{epsfig}
\usepackage[T1]{fontenc}
\usepackage{color}
\usepackage{verbatim}
\usepackage{appendix}
\usepackage{amsmath}
\usepackage{hyperref}
\usepackage[active]{srcltx}
\usepackage{xypic}
\usepackage{amscd}
\usepackage{color}
\newfont{\sheaf}{eusm10 scaled\magstep1}

\newtheorem{thm}{Theorem}[section]

\newtheorem{lemma}[thm]{Lemma}

\newtheorem{proposition}[thm]{Proposition}

\theoremstyle{definition}
\newtheorem{remark}[thm]{Remark}

\newtheorem{definition}[thm]{Definition}

\DeclareMathOperator{\Supp}{Supp}
\DeclareMathOperator{\Sing}{Sing}
\DeclareMathOperator{\Aut}{Aut}

\DeclareMathOperator{\Hom}{Hom}
\DeclareMathOperator{\Pic}{Pic}

\def\c1{\operatorname{c_1}}
\def\c2{\operatorname{c_2}}

\def\Sym{\operatorname{Sym}}

\def\c{\mathfrak{P}}

\def\PP{{\mathbb P}}

\def\cong{\simeq}

\def\+{\oplus}                   
\def\*{\otimes}                  

\def\Aut{\operatorname{Aut}}

\def\Hom{\operatorname{Hom}}

\def\Pic{\operatorname{Pic}}
\def\Supp{\operatorname{Supp}}

\def\Supp{\operatorname{Supp}}
\def\det{\operatorname{det}}
\def\Sing{\operatorname{Sing}}

\def\Prym{\operatorname{Prym}}

\title[ ]{  Edge and Fano on Nets of Quadrics}
\author{ Alessandro Verra}

\address{A. Verra \\  Dipartimento di Matematica e Fisica \\ Universit\'a Roma Tre \\ Italy }
 
 \email{sandro.verra@gmail.com}

\thanks{Partially supported by PRIN Project \it Geometry of Algebraic Varieties \rm and by INdAM-GNSAGA  }

\begin{document}



\maketitle
\begin{abstract} 
{In a number of papers by Edge, and in a related paper by Fano, several properties are discussed about the family of scrolls $R \subset \mathbb P^3$ of degree 8 whose plane sections are
projected bicanonical models of a genus $3$ curve $C$. This beautiful classical subject is implicitely related to the moduli of semistable rank two vector bundles on $C$ with bicanonical determinant. In this paper such a matter is reconstructed in modern terms from the modular point of view. In particular, the stratification of the family of scrolls $R$ by $\Sing R$ is considered and  the cases where $R$
has multiplicity $\geq 3$ along a curve are described.}
\end{abstract}
\section{Introduction}
In a series of five papers, sharing the name \it Notes on a Net of Quadric Surfaces\rm, Edge was gathering in a unique scientific work the fundamentals
of the theory of nets of complex projective quadric surfaces and of their classification, from the point of view of Invariant Theory.  These papers, 
appeared between 1937 and 1942, go far beyond the basic properties and, somehow, interestingly relate to modern theory of vector bundles on
curves and their associated moduli spaces. \par 
A standard projective invariant of a general net $N$ of quadrics of $\mathbb P^{d-1}$ is its discriminant, that is the degree $d$ plane curve 
$C \subset N$ parametrizing singular quadrics. As is well known, the locus in $\mathbb P^d$ of the singular points of the quadrics parametrized by $C$ is a 
smooth model of $C$ embedded by $\omega_C \otimes \theta$, where $\theta$ is a non effective theta characteristic on $C$. Moreover $N$ can
be uniquely reconstructed from the pair $(C, \theta)$ up to projective equivalence. This follows from a classical  result of A. C. Dixon on the representations of 
the equation of a plane curve as a  determinant of a symmetric matrix of linear forms,  see \cite{Di} and \cite{D}  p. 187. \par
In his \it Notes \rm Edge studies net of quadrics in $\mathbb P^3$, therefore he deals with plane quartics $C$ and their non effective theta
characteristics. The most beautiful geometry behind a net of quadric surfaces, and its connections to the theory of rank two vector bundles on
the discriminant quartic $C$,  is specially revealed in his paper III: \it The Scroll of Trisecants of the Jacobian curve, \cite{E-III}. \rm 
This paper studies the embedding of $C$ in $\mathbb P^3$ defined by the line bundle $\omega_C \otimes \theta$ and, in particular,  the scroll $R_{C,\theta}$,
union of the trisecant lines to $C$. \par $R_{C, \theta}$ is a classical object of study from several points of view.  For instance it is related to theta characteristics on plane quartics, to the Scorza map
and to bilinear Cremona transformations, cfr. \cite{D} chapters 5, 6 and 7. However, to proceed with order and motivate the title of our note, let us continue by saying that Fano proved to be an attentive  reader
of Edge's \it Notes\rm. In particular he wrote \it Su alcuni lavori di W. L. Edge \rm where, not without some criticism,  $R_{C, \theta}$ and nets of quadrics
are considered as very special cases of a more general situation, \cite{F}. \par 
His point of view is that $R_{C,\theta}$ should be considered as a special member of a wider family of octic scrolls $R$ in $\mathbb P^3$, which deserves to be investigated. This is the family of the 
scrolls $R$ in $\mathbb P^3$ whose smooth minimal model is the projectivization of a rank two vector bundle $E$ over $C$ with bicanonical determinant.  \par     In these pages we discuss the content of
Fano's paper and its special connections to nets of quadrics and Edge's work.    Of course never the word 'vector bundle' appears in the writings of both Edge and Fano. However Fano's point of view and their study of octic scrolls $R$ should be considered as a picture  'ante litteram' of the moduli space of semistable vector bundles of rank two and bicanonical determinant on a smooth non hyperelliptic curve $C$ of genus $3$. Actually this picture admits several  generalizations nowadays  in the literature, cfr. \cite{BV, HH, NR, Ty}. \par
What makes the difference between a general $R$ and $R_{C, \theta}$? How to describe the specializations of $\Sing R$? And how to recover from $R_{C,\theta}$ a net quadrics?
We are going to follow closely the geometric ideas of Edge and Fano, in order to describe their beautiful answers in modern terms. We partially rebuild and prove by modern tools
Fano's comments and results as follows. \par  Let $E$ be a globally generated, stable rank two vector bundle with bicanonical determinant on a non hyperelliptic curve $C$ of genus $3$ and let $\overline E$ be its dual span, that is the dual of the  Kernel of the evaluation $H^0(E) \otimes \mathcal O_C \to E$. Let $R \subset \mathbb P^3$ be the scroll defined by the tautological map of $E$: 
\begin{thm}  Let the Jacobian of $C$ be of Picard number one. Then a component of $\Sing R$ has multiplicity $m \geq 3$ iff one of the following  conditions holds.
 \medskip \par \underline {$m = 3$} \it One has $\overline E \cong E \cong T_{\mathbb P^2 \vert C} \otimes \theta$, $\theta$  is a non effective theta characteristic. $R$ is the scroll of trisecant lines to $\Sing R$ and this is
  $C$ embedded in $\mathbb P^3$ by $\vert \omega_C \otimes \theta \vert$. 
\medskip \par  \underline{$m = 4$} $\overline E$ is strictly semistable. $\Sing R$ is a rational normal cubic in $\mathbb P^3$. Moreover each line in $R$ is bisecant to $\Sing R$.

\end{thm}

\section{Quadrics through the bicanonical curve $C$ in $\mathbb P^5$}
Fano's point of view is well represented by the title of this section. To describe it let us consider the bicanonical embedding $C \subset \mathbb P^5$ of a smooth plane quartic and its ideal sheaf
$\mathcal I_{C \vert \mathbb P^5}$. Then, as is well known, we have a unique embedding 
$$
C \subset V \subset \mathbb P^5
$$
where $V$ is a Veronese surface. In particular the net of conics of $V$ cuts on $C$ the canonical linear system $\vert \omega_C \vert$. Furthermore the ideal of $C$ is generated by quadrics and
$C$ is a quadratic section of $V$. We fix the notation $\mathcal B$ for the moduli space of semistable rank two vector bundles $E \to C$ of bicanonical determinant and $[E]$ for the moduli point of $E$. Let $\mathbb P^6 := \vert \mathcal I_{C \vert \mathbb P^5}(2) \vert$, a natural rational map 
$$
\gamma: \mathcal B \to \mathbb P^6,
$$
generically finite of degree two, is provided by the next construction, see \cite{BV}. Let $[E] \in \mathcal B$ be a stable and general point, then $E$ is globally generated and $h^0(E) = 4$. Moreover the evaluation map $ev$  for global sections defines the exact sequence
$$
0 \to \overline E^* \to H^0(E) \otimes \mathcal O_C \stackrel{ev}  \to E \to 0
$$
so that $\overline E$ is stable. Let $f: \mathbb P \to \mathbb P^3 := \mathbb PH^0(E)^*$ be the tautological map of the projectivization $\mathbb P$
of $E$ and let $R := f(\mathbb P)$. It is also known that, for a general point $[E]$, the rational map $f: \mathbb P \to R$ is a birational morphism.  Notice that
$$
R = \bigcup_{x \in C} R_x
$$
where $R_x := \mathbb PE^*_x$ and $E^*_x \subset H^0(E)^*$ is the linear embedding induced by $ev_x$. Then the next property is a standard conclusion for a general $[E]$.
\begin{lemma} $R$ is an octic scroll in $\mathbb P^3$ whose general plane section is a linear projection of
the bicanonical model of $C$. Moreover $h^0(\mathcal O_R(1)) = h^0(E) = 4$.
\end{lemma}
Furthermore  the determinant map $det: \wedge^2 H^0(E) \to H^0(\omega^{\otimes 2}_C)$ is surjective for such an $[E]$. On the other hand let us consider the Pl\"ucker embedding $$ G \subset \mathbb P^5 := \mathbb P(\wedge^2 H^0(E)^*)$$  of the Grassmannian of lines of $\mathbb P^3 = \mathbb PH^0(E)^*$. The next lemma also follows.
\begin{lemma} The bicanonical embedding of $C$ factors through the map sending $x$ to the parameter point of the line $R_x$ and the Pl\"ucker embedding of $G$. \end{lemma}
Let $\mathcal E^*$ be the universal bundle of $G$ and let
$$
0 \to \overline{\mathcal E}^*  \to H^0(\mathcal E) \otimes \mathcal O_G \to \mathcal E \to 0
$$
be the exact sequence of the universal and quotient bundles. These vector bundles are uniquely defined, via Serre's construction, by the
two connected components of the Hilbert scheme of planes in the quadric $G$. Moreover one has $\mathcal E \otimes \mathcal O_C \cong E$ and $\overline {\mathcal E} \otimes \mathcal O_C \cong \overline E$. Using this we can finally define the above mentioned map $\gamma$.
\begin{definition} \it $\gamma: \mathcal B \to \mathbb P^6$ is the map sending $[E] \in \mathcal B$ to the quadric $G$. \end{definition}
The fibre of $\gamma$ at $G$ consists of two points, labeled by the spinor bundles $\mathcal E$ and $\overline {\mathcal E}$ of $G$. Fano was of course aware of the relations between
quadrics $G$ through the bicanonical model of $C$ and the corresponding moduli space of octic scrolls $R$ in $\mathbb P^3$ whose minimal desingularization is the projectivization of a \it stable \rm $E$. Edge is aware as well but, in his investigations on net of quadrics, he is only considering those special scrolls which are associated to nets of quadrics. Fano has a list of three cases of scrolls $R$ with special 
properties: (a), (b), (c). Cases (a) and (b) concern special features of $\Sing R$ and (a) involves nets of quadrics. Case (c) is of different nature, though related to  (a) and (b), and concerns L\"uroth quartics. Since our principal interest is here $\Sing R$ we treat only cases (a) and (b). \par 
In order to describe special properties of $R$ it is natural to observe that $\Sing R$ is non empty. One expects a curve of
double points having finitely many triple points, triple for $R$ as well. This is in some sense Fano's starting point. A simple vocabulary translates these properties of $R$ in properties of $G$ with respect to the varieties of bisecant lines and of trisecant planes to $C$. This is useful to understand the special cases and intimately related to the modern description of the theta map for the moduli space $\mathcal B$. We address this matter in the next sections.
\section{The theta map $\gamma: \mathcal B \to \mathbb P^6$ and the Coble quartic}
As is well known one has $\Pic \mathcal B \cong \mathbb Z[\mathcal L]$, where $\mathcal L$ is ample. In the case of a smooth plane quartic we are considering, the linear system $\vert \mathcal L \vert$ defines an embedding
$$
\mathcal B \subset \mathbb P^7.
$$
$\mathcal B$ is a quartic hypersurface which is well known as a \it Coble quartic\rm, \cite{NR}. The map defining this embedding is known as the \it theta map \rm of $\mathcal B$. It is useful for our purposes to describe it as follows. Let $\Theta := C^{[2]}$ be the $2$-symmetric 
product of $C$, embedded in $J^2 := \Pic^2C$ by the Abel map, sending $d \in C^{[2]}$ to $\mathcal O_C(d)$. In other words $\Theta$ is just the natural Theta divisor
$$
\lbrace L \in J^2 \ \vert \ h^0(L) = 1 \rbrace.
$$
In particular $\Theta$ is a \it symmetric \rm Theta divisor, which is invariant by the involution $\iota: J^2 \to J^2$ sending $L$ to  $\omega_C \otimes L^{-1}$. As is well known  a
natural identification $$ \mathbb P^7 := \vert 2\Theta \vert $$ 
can be fixed, so that the embedding of $\mathcal B$ in $\vert 2\Theta \vert$ can be described as follows. Let $[E] \in \mathcal B$ then there exists a unique divisor $\Theta_E \in \vert 2\Theta \vert$ such that
$$
\Supp \Theta_E := \lbrace L \in J^2 \ \vert \ h^0(E \otimes L^{-1}) \geq 1 \rbrace,
$$
we will say that $\Theta_E$ is the theta divisor of $E$. It is well known that
$$
\mathcal B = \lbrace \Theta_E \ \vert \ [E] \in \mathcal B \rbrace.
$$
In particular let $[E_o] := [\omega_C \oplus \omega_C]$. Then $[E_o]$ is a semistable point of $\mathcal B$ and $\Theta_{E_o} = 2\Theta$. Now consider the 
standard exact sequence
$$
0 \to \mathcal O_{J^2}(\Theta) \to \mathcal O_{J^2}(2\Theta) \to \mathcal O_{\Theta}(2\Theta) \to 0.
$$
Passing to the associated long exact sequence, it follows that the restriction map 
$$
\rho: \vert 2\Theta \vert \to \mathbb P^6 := \mathbb PH^0(\mathcal O_{\Theta}(2\Theta))
$$
is the linear projection of center $[E_o]$. We fix the notation $E[k]$ for the standard vector bundle of rank $k$, over the $k$-symmetric product $C^{[k]}$, whose
fibre at $d \in C^{[k]}$ is $E_d := E \otimes \mathcal O_d$. Let us recall that $C$ is assumed to be non hyperelliptic and that $C^{[2]}$ is identified to the theta divisor $\Theta \subset J^2$ via the Abel map.
Then let us consider the map of vector bundles
$$
ev^{[2]}: H^0(E) \otimes \mathcal O_{C^{[2]}} \to E[2]
$$
defined by the evaluation of global sections. It is well known that the degeneracy scheme of $ev^{[2]}$ is the restriction of $\Theta_E$ to $\Theta$.
Moreover such a restriction is a curve if $[E] \neq [E_o]$. Assuming this we fix the following
\begin{definition} \it The degeneracy curve $B_E$ of $E$ is the restriction of $\Theta_E$ to $\Theta$. \end{definition}
Putting for simplicity $B_E = B$ we finally consider the correspondence
$$
I := \lbrace (d, z) \in \Theta \times \mathbb P^5 \ \vert \ z \in \ell_d \rbrace,
$$
where $\ell_d$ is the line spanned by the degree $2$ effective divisor $d \subset C \subset \mathbb P^5$. Let
$$
\begin{CD}
{\Theta}Ê@<{\alpha}<<Ê{I}Ê@>{\beta}>> {\mathbb P^5} \\
\end{CD}
$$
be the projection maps of $I$ and let $G$ be the smooth quadric defined as above by the vector bundle $E$. Then the next property is satisfied, \cite{BV} 4.22.
 \begin{lemma} $G$ is the unique quadric of $\mathbb P^5$ whose strict transform by $\beta$ is $\alpha^*B$. \end{lemma}
More geometrically the strict transform by $\beta$ of the linear system of quadrics $\vert \mathcal I_{C \vert \mathbb P^5}(2) \vert$ coincides with the
pull back of $\vert \mathcal O_{\Theta}(2\Theta) \vert$ by $\alpha$. Then, after the natural identification of these two linear systems, the next theorem follows.
\begin{thm} The map $\gamma: \mathcal B \to \mathbb P^6$ factors through the embedding
of $\mathcal B$ in $\vert 2\Theta \vert$ via the theta map and the linear projection $\rho: \vert 2\Theta \vert \to \mathbb P^6$.\end{thm}
See \cite{BV, NR}. It is now useful to summarize more informations from the literature on the embedding $\mathcal B \subset \vert 2\Theta \vert$ and the linear projection
$\gamma: \mathcal B \to \mathbb P^6$ of center $[E_o]$. Let us briefly recall that: \medskip \par
$\circ$ \it $\mathcal B$ is embedded as the unique quartic hypersurface, the Coble quartic of $C$, which is singular along the Kummer variety of the curve $C$ \rm
$$
K(C) := \lbrace [E] \in \mathcal B \ \vert \ [E] = [\omega_C(e) \oplus \omega_C(-e)], \ \mathcal O_C(e) \in \Pic^0(C) \rbrace.
$$
\par \par
$\circ$ \it $K(C)$ is $\Sing \mathcal B$, moreover it is the locus of strictly semistable points of $\mathcal B$. Each point of $K(C)$ has multiplicity two for $\mathcal B$.
\medskip \par
$\circ$ $[E] = [\omega_C(e) \oplus \omega_C(-e)]$ if and only if $E$ fits in an exact sequence \rm
$$
0 \to \omega_C(e) \to E \to \omega_C(-e) \to 0.
$$
\par
In order to describe $\gamma$ notice that a stable point $[E]$ satisfies $h^1(E) = 0$ and hence $h^0(E) = 4$. Indeed we have $H^1(E) = \Hom(E, \omega_C)$
by Serre duality, therefore $h^1(E) \geq 1$ implies that $[E]$ is strictly semistable. Notice that a point  $[E] \in K(C)$ satisfies $h^0(E) \geq 5$ iff $[E]Ê= [E_o]$, with $E_o = \omega_C \oplus \omega_C$. As
is well known the divisor $\Theta_{E_o}$ is just $2\Theta$. In particular the map $\gamma: \mathcal B \to \mathbb P^6$ of the Coble quartic is the linear projection from its double point $[E_o]$ and a rational double covering. Let  
$$
\iota: \mathcal B \to \mathcal B
$$
be its associated birational involution, since now we assume $[E] \neq [E_o]$.  Let $\ell$ be the line joining $[E_o]$ to $[E]$, then exactly one
  of the following cases is possible: 
\begin{enumerate} 
\item $\ell \subset \mathcal B$,
\item $\ell \cdot \mathcal B = 3[E_o] + [E]$ and $\iota([E]) = [E_o]$,
\item $\ell \cdot \mathcal B = 2[E_o]Ê+ [E]Ê+ [\iota(E)]$ and $\iota([E]) \neq [E_o]$.
\end{enumerate}
Now assume that $E$ is globally generated, then the usual sequence 
$$
0 \to \overline E^* \to H^0(E) \otimes \mathcal O_C \stackrel{ev} \to E \to 0
$$
is exact and either $\overline E$ is stable or $[\overline E]Ê= [E_o]$. Hence $\iota$ is regular at $[E]$ and $[\overline E]Ê= \iota([E])$. Let $\mathbb P$ and $\overline {\mathbb P}$   be
the projectivizations of $E$ and $\overline E$, then we have their tautological models $R \subset \mathbb P^3$ and $\overline R \subset \mathbb P^{3*}$. It is easy to see that, since the determinant
of $E$ is bicanonical, these are embedded as dual surfaces:
 $$
\overline R = \bigcup_{x \in C} \overline R_x,
$$
where $\overline R_x \subset \mathbb P^{3*}$ denotes the pencil of planes through the line $R_x$ of $R$.  
Let $G = \gamma([E])$, since $[E]Ê\neq [E_0]$ we have $h^0(E) = 4$ and $G$ can be constructed from $H^0(E)$ as in \cite{BV} section 2.

We recall this construction for a general [E]: for such an $[E]$ the determinant map $\wedge^2 H^0(E) \to H^0(\omega_C^{\otimes 2})$ is an isomorphism. Consider the
natural pairing  $$ \wedge^2 H^0(E) \times \wedge^2 H^0(E) \to \wedge^4 H^0(E) $$ and the quadratic form $q_E$ induced by it on the space $\mathbb P^5 = \mathbb P H^0(\omega_C^{\otimes 2})^*$.
By definition $G$ is the zero locus of $q_E$.  Notice also that $G$ is the Pl\"ucker embedding of the Grassmannian of lines of $\mathbb PH^0(E)^*$. Moreover $C$ is embedded in $G$ so that  $E^*$ is the restriction to $C$ of the universal bundle.  \par Let $\sigma_{2,0}$ be the family of planes of $G$ which are dual of the planes in $\mathbb P^3$, then each $P \in \sigma_{2,0}$ corresponds to 
some point $[s] \in \mathbb P^3$ so that $P \cdot C$ is the scheme of zeroes of $s \in H^0(E)$. In the same way let $\sigma_{1,1}$ be the family of planes of $G$ which are parametrizing the lines through a point of $\mathbb P^3$, then each $\overline P \in \sigma_{1,1}$ corresponds to a point $[\overline s] \in \mathbb PH^0(\overline E)$ and $\overline P \cdot C$ is the scheme of zeroes of $\overline s \in H^0(\overline E)$. \par
Let us also recall that $C \subseteq V \cap G$, where $V$ is the unique Veronese surface containing $C$.  The next theorem characterizes the locus of points $[E]$ such that $E$ is globally generated
and $\iota([E]) = [E_o]$.
\begin{thm} Let $E$ be stable and globally generated, then $V$ is contained in $G$ if and only if $\iota([E]) = [E_o]$. \end{thm}
 \begin{proof} We have $\overline E]Ê= \iota([E])$. Let $Sec \ V$ be the cubic hypersurface equal to the 
 
 union of the planes $\Pi$ spanned by the conics of $V$. Assume $V \subset G$, it is easily
 seen that $T = G \cap Sec \ V$ is union of planes $\Pi$. Now $Z := \Pi \cdot C$ is a canonical divisor. Since $\deg Z = 4$ and $E$ is stable it follows that $Z$ is 
 the scheme of zeroes of some $\overline s \in H^0(\overline E)$. Hence $\overline E$ is strictly semistable and it follows $[\overline E]Ê= [E_o]$. Conversely let
 $[\overline E] = [E_o]$, then there exists a plane $\Pi$ as above in $G$ so that $Z = \Pi \cdot C$ is a canonical divisor. But then we have $Z \subset B \subset \Pi \subset G$, where
 $B$ is a conic of $V$. Then $G$ contains $C \cup B$ and hence contains $V$ by counting the degrees.  \end{proof}

\begin{thm} Let $E$ be stable and globally generated then $R$ is an octic scroll. \end{thm}
\begin{proof} Since $E$ is globally generated, the tautological map $\tau: \mathbb P \to R$ is a morphism. Assume $\deg R \geq 2$, then either $R$ is a quartic scroll or a quadric and its minimal desingularization is the projectivization of a vector bundle $F \to A$, where $A$ is a smooth integral curve of genus $1$ or $0$.

Notice that $E = \pi^*F$, where $\pi: C \to A$ is a $2:1$ cover if $\deg R = 4$ or a $4:1$ cover if $\deg R = 2$.
We claim that $F$ is not stable. But then $E$ is not stable and this contradiction implies the statement. The claim is clear if $R$ is a quadric. If $R$ is a quartic then $R$ has a singular curve and hence
$F$ admits linear subbundles of degree two. Then $F$ is not stable.  \end{proof}
We assume from now the property that $E$ is stable and globally generated, which is the case to be considered.
  \section{Theta divisor of $E$ and singularities of $R$}

Now we want to study the degeneracy curve $B = \Theta \cdot \Theta_E$ and to see how $B$ is related to the singular locus of the scroll $R$ defined by $E$. Notice that we have
$$
\mathcal O_{\Theta}(B) \cong \mathcal O_{\Theta}(\sum (x_i+C)- \delta) \in \Pic \Theta
$$
cfr. \cite{BV} 4.7, where $\sum x_i$ is a bicanonical divisor and our notation is fixed as follows.  We set $(x + C) := \lbrace d \in \Theta \ \vert \ x \in \Supp d \rbrace$. Moreover $2\delta$ is the class of the 
diagonal $\Delta := \lbrace 2x \in \Theta, \ x \in C \rbrace$ and $\delta$ defines a natural $2:1$ cover $C \times C \to \Theta$.  \par
 \begin{lemma} $B$ is a smooth, integral curve of genus 19 intersecting transversally a general curve $(x + C)$ at six points. \end{lemma}
\begin{proof} Since the restriction $\rho: \vert 2\Theta \vert \to \vert \mathcal O_{\Theta}(2\Theta) \vert$ is surjective,  a general $[E] \in \mathcal B$ defines a general curve
$B = \Theta \cdot \Theta_E \in \vert \mathcal O_{\Theta}(2\Theta) \vert$. Since this is ample and base point free, $B$ is smooth and integral. 
Finally notice that $\vert \mathcal O_{J^2}(2\Theta) \vert$ is base point free and that its restriction to $C^{[2]}$ is $\vert B \vert$. Hence the restriction to $x + C$ of $\vert B \vert$ is also base point free,
which implies that a general $B$ is transversal to $x + C$.   \end{proof}
We denote by $\mathbb P_d$ the fibre of $\mathbb P \to C$ at $d \in \Theta$, then we have
$$
\mathbb P_d \subset \mathbb P E^*_d.
$$
If $d = x+y$ and $x \neq y$ then $\mathbb P_d$ is the union of the skew lines $\mathbb P_x, \mathbb P_y$ in the 3-dimensional space $\mathbb PE^*_d$. In what follows we
assume that $B$ is smooth, then  $h^0(E(-d)) = 1$ for any $d \in B$. Note that the evaluation $e_d: H^0(E) \to E_d$ defines a linear map of $3$-dimensional spaces
$ p_d: \mathbb PE_d^* \to \mathbb P^3$. Let $d \in B$, since $h^0(E(-d)) = 1$ it follows that the tautological map
$ \tau \vert \mathbb P_d: \mathbb P_d \to \mathbb P^3 $
fails to be an embedding exactly along a $0$-dimensional subscheme $$ s_d \subset \mathbb P_d $$ which is contracted by $\tau$ to a singular point of $R$. This defines a morphism
$$ \sigma: B \to \Sing R $$ sending $d \in B$ to the point $p_d(s_d)$. If $d = x + y$ with $x \neq y$ then $\sigma(d) = R_x \cap R_y$.  \par A more delicate question concerns the stratification of the family of 
scrolls $R$ according the properties of $\sigma$ and, in particular, according to its possible degrees. In \cite{F} Fano considers this problem and its special cases. Let us discuss something useful about.
Let $n: \mathbb P \to R$ be the normalization and $p: \mathbb P \to C$ the projection map. Assume $o \in \Sing R$ and consider the divisor $p_*n^*o = x_1 + \dots + x_k$. As usual let $[\overline E]$
$=$ $\iota([E]$, according to our convention we have 
$$
P_o \cdot C = x_1 + \dots + x_k,
$$
where $P_o$ is the plane in $G$ parametrizing all the lines through the point $o$.   
\begin{lemma} For every divisor $d \subset x_1 + \dots + x_k$ one has $h^0(F(-d)) \geq 1$. \end{lemma}
The proof follows immediately from the above remarks. Assume now that $\sigma$ is not birational, then we have $k \geq 3$ distinct lines of the scroll $R$, say $$R_{x_1} \dots R_{x_k},$$  passing through a general point $o \in \Sing R$. Let $P_o \subset G$ be the plane parametrizing the lines through $o$, then $x_1 \dots x_k$ are in $P_o$. Since $\mathcal I_{C \vert \mathbb P^5}$ is generated by quadrics no three of these points are collinear.  Since $E$ is stable and globally generated, we know that either $\overline E$ is stable or $[\overline E]Ê= [E_o]$. Therefore we have $k \leq 4$ and the equality holds iff $[\overline E] = [E_0]$. The next statement is known more in general.  We introduce a proof for completeness and  to study  its exceptions.
\begin{thm} $\sigma: B \to \Sing R$ is birational onto its image for a general $[E]$  and $\Sing R$ is an integral curve of multiplicity two for $R$. \end{thm}
 \begin{proof} Consider the open set $U := \lbrace t \in C^{[3]} \ \vert \ h^0(\omega_C(-t)) = 0 \rbrace$ and 
$$
Q := \lbrace (t, G) \in U \times \mathbb P^6 \ \vert \ P_t \subset G \rbrace.
$$
Since $C$ is generated by quadrics the projection $p: Q \to U$ is a $\mathbb P^3$-bundle. The fibre of $p$ at
$t$ is the linear system of all quadrics through $C \cup P_t$. Let $ q: Q \to \mathbb P^6 $ be the second projection. Since $Q$ is irreducible of dimension $6$ the fibre of $q$ at a general $t$ is 
a finite set.  Hence the smooth quadric $G$ associated to $[E]$ contains a finite set of planes $P_t$. Since $E$ and $\overline E$ are stable, no plane 
in $G$ of the same ruling of $P_t$  is $k$-secant to $C$ for $k \geq 4$. Let $o \in \Sing R$  then at most two lines $R_x, R_y$ of $R$ contain $o$. Hence
$\sigma^{-1}(o) = \lbrace x+y \rbrace $ and $\sigma: B \to \Sing R$ is birational. \end{proof}
 Actually the surjectivity of $q: Q \to \mathbb P^6$ is well known, see e.g. \cite{LN}. Equivalently the set of planes $P_t$, considered above,
 is finite and not empty for a general $[E]$. Thi set is strictly related to the scheme of triple points of $R$.
To study this scheme of triple points of $R$ let us consider $\overline E$ and the map of vector bundles
$$
ev^{[3]}: H^0(\overline E) \otimes \mathcal O_{C^{[3]}} \to \overline E[3].
$$
The degeneracy scheme of this map is expected to be of codimension three. It is essentialy the parameter space for the locus of points of multiplicity three in $R$.
\begin{definition} \it The degeneracy scheme of  $ev^{[3]}$ is the triple locus $T_E$ of $E$.   \end{definition}
$T_E$ has dimension zero for a general $[E]$. The count of the length of $T_E$ is of course possible, both by classical and modern methods. Notice
that, for a general $E$, the curve $\sigma(B) = \Sing R$ is the curve of double points of $R$ having degree $18$ and geometric genus $19$. Counting multiplicities
the length of $T_E$ is the number of triple points of $R$ and of $\sigma(B)$ as well. This can be computed via Cayley-Zeuthen formulae, see \cite{D} 10.4.7 and 10.4.9,
as Fano does:
\begin{lemma} If $E$ is general then $R$ has eight triple points. \end{lemma}
As it has been already emphasized, it is not forbidden that $T_E$ be a curve in special cases: even when $E$ is stable. This leads to the list of specializations
considered by Fano and to the beautiful interplay with nets of quadrics considered by Edge. We reconstruct this matter in the next section, via Coble quartic and modern 
tools from the theory of vector bundles on curves and their moduli.
 \section{Specializations: Scorza correspondence and nets of quadrics}
\it When the morphism $\sigma: B \to \Sing R$ is not birational? \rm In what follows we keep the assumption that $E$ is stable and globally generated. Under this assumption the next result, already pointed out by Fano, reveals beautiful geometry.
\begin{thm} Let $\sigma: B \to \Sing R$ be not birational, then two cases may occur: 
\medskip \par
\rm (a) \it $\iota([E]) \neq [E_o]$:  \it  $\sigma$ has degree $3$ and $\Sing R$ is the sextic embedding of $C$ defined by $\omega_C \otimes \theta$,
where $\theta$ is a non effective theta on $C$. \medskip \par
\rm (b)  $\iota([E]) = [E_o]$: \it $\sigma$ has degree $4$ and $\Sing R$ is a rational normal cubic.  
\medskip \par 
Moreover these are the only cases if the Jacobian of $C$ has Picard number one.
\end{thm}
Some remarks are in order.  At first notice that $R$ is not a cone: indeed this implies $h^0(E \otimes \omega_C^{-1}) \geq 1$ and $E$ unstable.
As was remarked before, $\iota$ is the birational involution of the Coble quartic $\mathcal B$, induced by the projection from $[E_o] \in \Sing \mathcal B$.   
\begin{lemma} Let $[F]Ê= \iota([E])$ then  $[F]$ $\neq$ $[E_o]$ $\Rightarrow$ $F$ is stable. \end{lemma}
\begin{proof} Let $[F]$ be not stable, then $[F] \in \Sing \mathcal B$ and the line containing the distinct points $[E_o], [E], [F]$ is in $\mathcal B$. Hence
$E$ is not globally generated: a contradiction. \end{proof}
Since $R$ is not a cone, a point $o \in \Sing R$ defines on $C$ the divisor $$ p_*n^*(o) := x_1 + \dots + x_k. $$
We can start our proofs of cases (a) and (b), adding something more, which is due in order to reveal some nice properties appearing.
\medskip \par
$\circ$ {\em Proof and more geometry on case (a)} \par
We assume here that $F$ is \it stable. \rm Then we have $k \leq 3$ for any $o \in \Sing R$ and we already know that $h^0(F(p_*n^*o)) \geq 1$. Finally we consider the reduced curve 
$$
\Gamma_k \subset \Sing R
$$
which is union of the irreducible components of $\Sing R$ having multiplicity $k$ for $R$. We assume that $\sigma$ is not birational, which implies that $\Gamma_3$ is not empty. 
Actually $\Gamma_3$ is the support of the triple locus $T_E$ of $E$, defined in 4.4. Let 
$$
a: \Gamma_3 \to Pic^3(C)
$$
be te map sending $t_o := p_*n^*o$ to $\mathcal O_C(T-o))$. Let $\tau \in \Pic^2(C)$ and
$$
\tau + C := \lbrace \tau(x), x \in C \rbrace \subset Pic^3(C).
$$
Assume that $\Pic^3(C)$ has Picard number one. Then the next theorem is the key to describe geometrically the case we are considering.
\begin{thm}  The image of $a$ is $\tau + C$, for a non effective $\tau \in \Pic^2(C)$. \end{thm}
\begin{proof} We claim that $a(\Gamma_3) \cdot (x+\Theta)$ has length $ \ell \leq 3$ for $x$ general. This implies $\ell = 3$, since the numerical class of $a(\Gamma)$ 
is in $\mathbb Z[ \frac{\Theta^2}2]$. Then, by Matsusaka criterion \cite{Ma}, $a(\Gamma)$ is $\tau + C$ for some $\tau \in \Pic^2C$. Finally $\tau$ is not effective: otherwise
we would have $h^0(F(-a-b-x)) \geq 1$, where $a,b \in C$ are fixed points and $\tau \cong \mathcal O_C(a+b)$. Since $R$ is not a cone it follows $R_a = R_b := L$. Let $o_x = L \cap R_x$, then we have a 
rational map $f: C \to L$ such that $f(x) = o_x$. Since $F$ is stable $f$ is injective: a contradiction. \par
To prove the claim consider the normalization $n: \mathbb P \to R$ and the projection $p: \mathbb P \to C$. Let $a_x: C \to \Pic^2(C)$ be the Abel map sending $y$ to $\mathcal O_C(x+y)$.
 Since $\deg \mathcal O_{x+C}(B) = 6$ the lenght of  $b_x := a^*_xB$ is $6$. Let $B_x := p^*b_x \cdot n^* \Sing R$ and $\Gamma_x = R_x \cdot \Sing R$, we have the diagram
$$
\begin{CD}
{\Gamma_x} @<n<< {B_x} @>p>>  {b_x}\\
\end{CD}
$$
Since $x$ is general and $\sigma$ birational, we can assume  that $p: B_x \to b_x$ is biregular. Otherwise we would have a length two divisor $d$ embedded in $B_x \cdot \mathbb P_y$
for some $y \in \Supp b_x$. Since $\sigma$ is an embedding on $d$, this implies $\sigma(R_x) = \sigma(R_y)$ and finally $R_x \subset \Sing R$, which is impossible for $x$.
Now consider $n: B_x \to \Gamma_x$. Since $F$ is stable, $n$ contracts a scheme $d \subset B_x$ to a point iff $h^0(F(-x-d)) \geq 1$ and $\deg d = 2$ that is iff $n(d) \in \Gamma_3$. Since
$\deg b_x = 6$ it follows that $n$ contracts at most $3$ disjoint schemes $d$. Hence $a(\Gamma) \cdot (x + \Theta)$ has length $\ell \leq 3$.  \end{proof}
 \medskip \par
Once established $a(\Gamma_3) = \tau + C$ we can reconstruct $\Theta_E = \Theta_F$ from $\tau + C$ and hence  the curve $B$. It is well known that
the cohomology class of the divisor
$$ T_{\tau} := \lbrace \tau(x-y), \ (x,y) \in C \times C \rbrace \subset \Pic^2(C), $$ 
is twice the class of the theta divisor. The same is true for the difference divisor 
$$
C - C := \lbrace N \in Pic^0(C), \  \vert N \cong \mathcal O_C(x-y), \ (x,y) \in C \times C \rbrace.
$$
This is actually the theta divisor in $Pic^0(C)$ of a special rank two vector bundle on the quartic model $C \subset \mathbb P^2$, namely $T_{\mathbb P^2 \vert C}(-1)$.
Indeed the Euler sequence
$$
0 \to \mathcal O_C(-1)\to H^0(\omega_C) \otimes \mathcal O_C \to T_{\mathbb P^2 \vert C}(-1)\to 0
$$
implies, tensoring by $\mathcal O_C(x-y)$ and passing to the long exact sequence,  that 
$$ C - C = \lbrace N \in \Pic^0(C) \ \vert \ h^0(T_{\mathbb P^2 \vert C} \otimes N^{-1}) \geq 1 \rbrace. $$ 
Let $\mathcal B_J$ be the moduli space of semistable vector bundles on $C$ of degree $8$ and $\mathbb P_J \to \Pic^0(C)$ the $\mathbb P^6$-bundle   with fibre at $e$ the linear
system $\vert \mathcal O_{\Theta}(2(\Theta+e) \vert$. Then the map $\gamma: \mathcal B \to \vert 2\Theta \vert$ immediately extends to the rational double covering
$$
\gamma_J: \mathcal B_J \to \mathbb P^6_J,
$$
where, for a general $[U] \in \mathcal B_J$,  the element $\gamma_J([U])$ is the curve $\Theta_U \cdot \Theta$ and $\Theta_U$ is the theta divisor in $\Pic^2(C)$  defined by the vector bundle $U$. The involution $\iota_J: \mathcal B_J \to \mathcal B_J$, induced by $\gamma_J$, is defined again by the standard exact sequence
$$
0 \to \overline U^* \to H^0(U) \otimes \mathcal O_C \stackrel{ev} \to U \to 0.
$$
Let $E_{\tau} := T_{\mathbb P^2 \vert C} \otimes \tau^{-1}$, then $E_{\tau}$ is stable. Notice that $T_{\tau}$ is the $2$-Theta divisor associated to $E_{\tau}$. This just follows because the condition $N \cong \tau(x-y)$
is equivalent to $h^0(E_{\tau} \otimes N^{-1}) \geq 1$.  In particular it follows that $\gamma_J([E_{\tau}]) = T_{\tau}$.

\begin{thm} The vector bundles $E_{\tau}$, $E$ and $F$ are isomorphic and $\tau$ is a non effective theta characteristic on $C$. \end{thm} 
\begin{proof} For a general $\tau$, and hence for any $\tau$, it is easy to check that 
$$
T_{\tau} \cdot \Theta = \lbrace d \in \Theta \ \vert \ d+y \in \vert \tau(x) \vert, \ (x,y) \in C^2 \rbrace = \lbrace d \in \Theta \ \vert \ h^0(F(-d)) \geq 1 \rbrace = B.
$$
Then $[E_{\tau}]$ is in the fibre of $\gamma_J: \mathcal B_J \to \mathbb P^6_J$ at $B$ and it is isomorphic to $E$ or $F$. This implies $\omega_C^{\otimes 3} \otimes \tau^{-2} \cong \det E_{\tau} \cong \omega_C^{\otimes 2}$ and hence that $\tau$ is a theta characteristic. Finally  it is well known that $\iota([E_{\tau}]Ê= [E_{\tau}]$, which implies the statement. \end{proof}
More in general let $[E] \in \mathcal B_J$, where $E$ is stable and globally generated and $\Pic^2(C)$ has Picard number one. The same arguments used above imply that: \medskip \par
\centerline {\it $R$ has multiplicity $\geq 3$ along a curve $\Leftrightarrow$ $E \cong T_{\mathbb P^2 \vert C} \otimes \tau^{-1}$} \medskip \par
where $R$ is tautological model of $\mathbb  P = \mathbb P(T^*_{\mathbb P^2 \vert C} \otimes \tau)$ and $\tau$ satisfies $h^0(\tau) = 0$. We cannot avoid to sketch very briefly the classical properties of $R$, cfr. \cite{D} 5.5 and 6. Keeping our notation, we have the following description.  Consider the curves: \medskip \par
$\circ$  $T_E = \lbrace t \in C^{[3]} \ \vert \ t \in \vert \tau(x) \vert \ \text{ \it for some $x \in C$} \rbrace$, \par \it
$\circ$  $B = \lbrace d \in \Theta \ \vert \ d \subset t, \ \text {for some $t \in T_E$} \rbrace$. \par \rm
\medskip \par
Notice that, since $h^0(\tau) = 0$, $T_E$ is biregular to $C$ via the map $t \to \mathcal O_C(t-\tau)$. \par Let $t = a+b+c \in T_E$. Since $E = F = E_{\tau}$, we know that $h^0(E_{\tau}(-t)) \geq 1$ and hence that the lines $R_a, R_b, R_c$ intersect in one point $o_t := R_a \cap R_b \cap R_c$. Let
$$
f: C \to \mathbb P^3
$$
be the map $t \to o_t$: it turns out that $f$ is the embedding defined by $\omega_C \otimes \tau$. Its image $\Gamma$ is a curve of degree $6$ and multiplicity $3$ for $R$. On
the other hand we know that, for a general $E$, $\Sing R$ is instead a double curve of degree $18$. This implies, for degree reasons, that $f(C) = \Sing \Gamma$. Recall also that, for every
$x$, the scheme $(x+C) \cdot B$ has length $6$ and defines in $C$ a divisor $b_x =\sum_{1 \dots 3} y_i + z_i$ such that, up to reindexing, $h^0(E_{\tau}(-x-y_i-z_i)) = 1$, $1 \leq i \leq 3$. Equivalently  the line $R_x$ of $R$ is trisecant to the curve $\Gamma$. Under the previous assumptions  we conclude that:

 \medskip \par
\centerline {\it $E \cong T_{\mathbb P^2 \vert C} \otimes \tau^{-1}$ $\Leftrightarrow$ $R$ is the scroll of the trisecant lines to $\Gamma$}.
\medskip \par
Of course, as very well considered by Edge in his series of papers, if $\tau$ is a non effective theta characteristic the $\Gamma$ is the locus of the singular points of the net of quadrics
defined by the pair $(C, \tau)$. Notice also that the trisecant locus $T_E$ of this case gives rise to the well known and interesting  Scorza's correspondence $$ S(\theta) = \lbrace (x,y) \in C \times C \ \vert \ h^0(\theta(x-y)) = 1 \rbrace $$
and its related topics, see \cite{DK}. 
 \medskip \par
{\em Proof and geometric description of case (b)} \par
Now we are left to discuss the case $[E_0] = \iota([E])$. This means that that the line joining $[E_0]$ to $[E]$ is an inflexional tangent to $\mathcal B$ at $[E_0]$ not contained in $\mathcal B$,
see section 3. Let us conclude very briefly by this case, which is remarked by Fano in \cite{F} as an additional case to Edge's analysis of nets of quadrics. Actually it representative of the net of 
quadrics whose base locus is a rational normal cubic. Since $[E_o] = [\omega_C^2]$, it follows that $F = \overline E$, fits in an exact sequence
$$
0 \to \omega_C \to F \to \omega_C \to 0.
$$
Moreover we know from theorem that the quadric $G$ defines by $[E]$ has rank six, since $F$ and $E$ are not isomorphic. With our conventions $G$ is the Grassmannian of lines of
$\mathbb P^3 = \mathbb PH^0(E)^*$ and the ruling of the dual of the planes of $\mathbb P^3$ corresponds to the planes of class $\sigma_{2,0}$. A non zero section $s \in H^0(E)$
defines a plane $P$ of this class and its scheme of zeroes $P \cdot C$ has always $\leq 3$.  Furthermore we know from theorem  3.4  that $G$ contains the Veronese $V$ surface 
already considered. Also we know that  the secant variety $ Sec \ V$ of $V$ cuts on $G$ a threefold $T$ which is union of planes $\overline P \subset G$. Each $\overline P$
contains a conic of $V$ and $\overline P \cdot C$ is a canonical divisor. Since its degree is $4$, it follows that the planes $\overline P$ have the opposite class $\sigma_{1,1}$.
Asnis well known the class of $V$ in $G$ is then $3\sigma_{2,0} + \sigma_{1,1}$. Equivalently the point of $V$ parametrize the bisecant lines to a skew cubic
$$
\Gamma \subset \mathbb P^3.
$$
Each point $o \in \Gamma$ defines the plane $\overline P_o$ parametrizing the lines through $o$. The cone of vertex $o$ containing $\Gamma$ is parametrized by a conic
$B \subset V$., moreover it is now clear that $b_o := B \cdot C$ is a canonical divisor of $C$. Counting multiplicities, $b_o$ parametrizes four lines in the scroll $R$ passing
through the point $o$. This explains, up to further details, why $R$ has multiplicity four along $\Gamma$ and completely describes this scroll.
\medskip \par
To conclude, let us say that papers of Edge and Fano we discussed in this article, though dedicated to important and known classical topics,  still are rich of arguments to be investigated.
Our note is a small contribution to support this point of view.

\end{document}